\documentclass{amsart}
\usepackage{amsmath,amsthm}
\usepackage{amsfonts,amssymb}
\usepackage{enumerate}
\usepackage{accents,color}
\usepackage{graphicx}

\newcommand{\lvt}{\left|\kern-1.35pt\left|\kern-1.3pt\left|}
\newcommand{\rvt}{\right|\kern-1.3pt\right|\kern-1.35pt\right|}

\newtheorem{thm}{Theorem}[section]
\newtheorem{cor}[thm]{Corollary}

\newtheorem{prop}[thm]{Proposition}

\theoremstyle{remark}

 \def\a{{\alpha}}
 \def\b{{\beta}}
 
 \def\k{{\kappa}}
 \def\t{{\theta}}
 \def\l{{\lambda}}
 \def\d{{\delta}}
 
 \def\s{\sigma}
 \def\la{{\langle}}
 \def\ra{{\rangle}}

 \def\CH{{\mathcal H}}

 \def\CT{{\mathcal T}}
 
 \def\CV{{\mathcal V}}

 \def\BB{{\mathbb B}}

 \def\RR{{\mathbb R}}
  \def\SS{{\mathbb S}}

 \def\lb{{\boldsymbol{\large {\l}}}} 
 \def\one{{\mathbf 1}}
 
 \def\proj{\operatorname{proj}}

\newcommand{\wt}{\widetilde}
\newcommand{\wh}{\widehat}

\def\f{\frac}

\graphicspath{{./}}

\begin{document}

\title {An integral identity with applications in orthogonal polynomials}

\author{Yuan Xu}
\address{Department of Mathematics\\ University of Oregon\\
    Eugene, Oregon 97403-1222.}\email{yuan@uoregon.edu}

\date{\today}
\thanks{The work was supported in part by NSF Grant DMS-1106113}

\begin{abstract}
For $\boldsymbol{\large {\lambda}} = (\lambda_1,\ldots,\lambda_d)$ with $\lambda_i > 0$, it is proved that 
\begin{equation*}
    \prod_{i=1}^d  \frac{ 1}{(1-  r x_i)^{\lambda_i}} =  \frac{\Gamma(|\boldsymbol{\large {\lambda}}|)}{\prod_{i=1}^{d} \Gamma(\lambda_i)} \int_{\CT^d} \frac{1}{ (1- r \la x, u \ra)^{|\boldsymbol{\large {\lambda}}|}} \prod_{i=1}^d u_i^{\lambda_i-1} du,
\end{equation*}
where $\CT^d$ is the simplex in homogeneous coordinates of $\RR^d$, from which a new integral relation for 
Gegenbuer polynomials of different indexes is deduced. The latter result is used to derive closed formulas for 
reproducing kernels of orthogonal polynomials on the unit cube and on the unit ball. 
\end{abstract}

\maketitle

\section{Introduction}
\setcounter{equation}{0}

Let $\RR_+^d = \{ x \in \RR^d: x_1 \ge 0, \ldots, x_d \ge0\}$ be the positive quadrant of $\RR^d$. Let $T^{d-1}$ be the
simplex in $\RR^{d-1}$ defined by $T^{d-1} : = \{ y \in \RR_+^{d-1}: |y| \le 1\}$, where $|y| := y_1 + \cdots + y_{d-1}$. 
Written in homogeneous coordinates, this simplex is equivalent to 
$$
       \CT^{d} : = \{ y \in \RR_+^{d}: |y| =1\}.
$$
Observe that $\CT^2$ reduces to the interval $[0,1]$. 
The main result in this paper is the following integral identity and its applications. 

\begin{thm} \label{thm:main}
Let $d =2,3,\ldots$ and $\lb = (\l_1,\ldots, \l_d)$ with $\l_i > 0$, $1 \le i \le d$. For $x \in \RR^d$ and $r \ge 0$ such that 
$r|x_i| \le 1$, $1 \le i\le d$,  
\begin{equation} \label{eq:main}
    \prod_{i=1}^d  \frac{ 1}{(1-  r x_i)^{\l_i}} =  \frac{\Gamma(|\lb|)}{\prod_{i=1}^{d} \Gamma(\l_i)} \int_{\CT^d} \frac{1}{ (1- r \la x, u \ra)^{| \lb|}} \prod_{i=1}^d u_i^{\l_i-1} du.
\end{equation}
\end{thm}

Although the identity \eqref{eq:main} is elementary, it leads to new identities on the Gegenbauer polynomials that have
interesting applications to orthogonal polynomials in several variables. For $\l > -1/2$, define 
$$
  w_\l(t) := (1-t^2)^{\l -1/2}, \qquad -1 < t < 1. 
$$
The Gegenbauer polynomial $C_n^\l$ is defined as the orthogonal polynomial of degree $n$ with respect to $w_\l$, 
normalized by $C_n^\l(1) = \binom{n+2\l -1}{n}$. It satisfies the relation
$$
   c_\l \int_{-1}^1 C_n^\l(t) C_m^\l(t) w_\l(t) dt = h_n^\l \d_{n,m}, \qquad h_n^\l := \frac{\l}{n+\l} C_n^\l(1), 
$$
where $c_\l$ is the normalization constant of $w_\l$. Let $\wt C_n^\l(t): =C_n^\l(t) / \sqrt{h_n^\l}$. Then $\wt C_n^\l$ is the $n$-th orthonormal polynomial with respect
to $w_\l$. It follows readily that 
$$
    Z_n^\l (t) := \wt C_n^\l(1)\wt C_n^\l(t) = \frac{n+\l}{\l} C_n^\l(t). 
$$ 
The Gegenbauer polynomials satisfy the following generating relations: for $0 \le r < 1$,
\begin{equation}\label{eq:generatingC}
 \f{1}{(1-2 r t + r^2)^{\l}} = \sum_{n=0}^\infty C_n^\l(t) r^n \quad \hbox{and}\quad
  \f{1-r^2}{(1-2 r t + r^2)^{\l+1}} = \sum_{n=0}^\infty  Z_n^\l(t) r^n.
\end{equation}

One application of Theorem \ref{thm:main} is a closed formula for the reproducing kernels of product Gegenbauer
polynomials on the cube $[-1,1]^d$, which will be discussed in Subsection 3.2. Another application of 
Theorem \ref{thm:main}, more directly, gives two new identities for the Gegenbauer polynomials. 

\begin{thm} \label{thm:Gegen}
For $\l > -1/2$ and $\mu > 0$, 
\begin{equation} \label{eq:Gegen-1}
   C_n^\l(x) = c_\mu \s_{\l,\mu} \int_{-1}^1 \int_0^1 C_n^{\l+\mu} (s x +(1-s) y) s^{\l-1} (1-s)^{\mu-1}ds\, w_\mu(y) dy,
\end{equation}
and, furthermore, 
\begin{equation} \label{eq:Gegen-2}
 Z_n^\l(x) = c_\mu \s_{\l+1,\mu} \int_{-1}^1 \int_0^1 Z_n^{\l+\mu} (s x +(1-s) y) s^{\l} (1-s)^{\mu-1} ds\, w_\mu(y)dy,
\end{equation}
where
$$
  \s_{\l,\mu} := \frac{\Gamma( \l+\mu)}{\Gamma(\l) \Gamma(\mu)} \quad \hbox{and} \quad c_\mu:= 
      \frac{\Gamma( \mu+1)}{\Gamma(\f12) \Gamma(\mu+\f12)}.
$$
\end{thm}

These identities are new. It is known in the literature (see, for example, \cite[p. 25]{A}) that 
if $\mu > 0$ and $\l > -1/2$, then 
$$
     \frac{C_n^{\l+\mu}(x)}{C_n^{\l+\mu}(1)} = \int_{-1}^1   \frac{C_n^\l(y)}{C_n^\l (1)} d \mu_{x}(y), \qquad -1 \le x \le 1,
$$
where $d\mu_x(y)$ is strictly positive and absolutely continuous when $-1 < x<1$, and it is 
a unit mass at $y =x$ when $x^2 =1$. In comparison, the index of the Gegenbauer polynomial in the 
left hand side of the identity \eqref{eq:Gegen-1} is smaller than the index appeared in the right hand side. 

The new identities have interesting applications. Since $Z_n^\l(\la x,y\ra)$, when $\l =\f{d-2}{2}$, 
is the reproducing kernel of spherical harmonics of degree $n$ on the unit sphere $\SS^{d-1}$, the identity 
\eqref{eq:Gegen-2} can be used to derive a closed fromula for the reproducing kernels of orthogonal 
polynomials with respect to $W_{\l,\mu}(x) =\|x\|^{2\l}(1-\|x\|^2)^{\mu-\f12}$ on the unit ball, which is known
only in the case $\lambda =0$ previously. Such a formula plays an essential role for studying Fourier orthogonal 
expansions. We shall use it, as an application, to determine the critical index for Ces\`aro means of orthogonal 
expansion with respect to $W_{\l,\mu}$ in Subsection 3.3.

The paper is organized as follows. We prove the main theorems in the following section and discuss applications 
of the main results in orthogonal polynomials of several variables in Section 3. 

\section{Proof of Theorems \ref{thm:main} and \ref{thm:Gegen}}
\setcounter{equation}{0}
 
The proof of Theorem 1 uses multinomial theorem. 

\medskip\noindent
{\it Proof of Theorem \ref{thm:main}}.
Observe that
\[   |r \langle x,u\rangle | \leq r \sum_{i=1}^d |x_i| u_i < \sum_{i=1}^d u_i = 1, \]
so that one can use the multinomial theorem
\[    (1-s_1-\cdots-s_d)^{\nu} = \sum_{n_1=0}^\infty \cdots \sum_{n_d=0}^\infty \frac{(-\nu)_{|\mathbf{n}|}}{n_1!\ldots n_d!} s_1^{n_1}\cdots s_d^{n_d}  \]
(see, e.g., \cite[Eq. (220) on p.~329]{SriKar}). We then have
\begin{align*}
 &\int_{\mathcal{T}^d} \frac{1}{(1-r\langle x,u \rangle)^{|\boldsymbol{\lambda}|}} \prod_{i=1}^d u_i^{\lambda_i-1}\, du  \\
  & \qquad\qquad
   = \sum_{n_1=0}^\infty \cdots \sum_{n_d=0}^\infty \frac{(|\boldsymbol{\lambda}|)_{|\mathbf{n}|}}{n_1!\cdots n_d!} r^{|\mathbf{n}|} x_1^{n_1}\cdots x_d^{n_d}
    \int_{\mathcal{T}^d} \prod_{i=1}^d u_i^{\lambda_i+n_i-1}\, du . 
\end{align*}
The integral is now a multivariate beta integral
\begin{align*} 
\int_{\mathcal{T}^d} \prod_{i=1}^d u_i^{\lambda_i+n_i-1}\, du = & 
  \int_{T^{d-1}} u_1^{\lambda_1+n_1-1}\cdots u_{d-1}^{\lambda_{d-1}+n_{d-1}-1}
     (1-|u|)^{\lambda_d+n_d-1} \, du_1\cdots du_{d-1} \\
  = &  \frac{\Gamma(\lambda_1+n_1)\cdots \Gamma(\lambda_d+n_d)}{\Gamma(|\boldsymbol{\lambda}|+|\mathbf{n}|)}, 
\end{align*}
so that
\begin{align*}
 \int_{\mathcal{T}^d} & \frac{1}{(1-r\langle x,u \rangle)^{|\boldsymbol{\lambda}|}} \prod_{i=1}^d u_i^{\lambda_i-1}\, du  \\
   & \qquad = \sum_{n_1=0}^\infty \cdots \sum_{n_d=0}^\infty \frac{(|\boldsymbol{\lambda}|)_{|\mathbf{n}|}}{\Gamma(|\boldsymbol{\lambda}|+|\mathbf{n}|)}
  \frac{\Gamma(\lambda_1+n_1)\cdots \Gamma(\lambda_d+n_d)}{n_1!\cdots n_d!} r^{|\mathbf{n}|} x_1^{n_1}\cdots x_d^{n_d}.
\end{align*}
Now $(|\boldsymbol{\lambda}|)_{|\mathbf{n}|} = \Gamma(|\boldsymbol{\lambda}|+|\mathbf{n}|)/\Gamma(|\boldsymbol{\lambda}|)$, hence the multiple sum
factors into $d$ single sums, and by the binomial theorem
\[    \sum_{n=0}^\infty \frac{(\lambda)_n}{n!} (rx)^n = \frac{1}{(1-rx)^{\lambda}}, \qquad |rx| < 1, \]
the result in \eqref{eq:main} follows. 
\qed
 
\medskip

The above proof is communicated to us by Walter Van Assche. Our original proof uses the generalized 
Gegenbauer polynomials that are orthogonal with respect to the weight function 
$$
  w_{\lambda, \mu} (x) := |x|^{2\mu} (1-x^2)^{\lambda - \frac12}, 
     \qquad -1 \le x \le 1, \qquad \lambda, \mu > -1 / 2.
$$
Let $D_n^{(\lambda, \mu)}$ denote the orthonormal polynomial of degree $n$ with respect to $w_{\l,\mu}$, 
\begin{equation}\label{eq:genGegen}
c_{\lambda, \mu} \int_{-1}^1 D_n^{(\lambda, \mu)}(x)D_m^{(\lambda, \mu)}(x) 
   w_{(\lambda, \mu)} (x) d x  = \delta_{n,m},
\end{equation}
where $c_{\l, \mu}$ is the constant defined by $c_{\l,\mu} \int_{-1}^1 w_{\l,\mu} (x) dx =1$. 
The following identity is used in the original proof and it will also be needed in Section 3.3. 

\begin{prop}
For $\lambda > 0$ and $\mu >0$, 
\begin{align} \label{eq:addition}
& C_n^{\lambda+\mu}(\cos \theta \cos \phi \, t  + \sin \theta \sin \phi \,s)
  =   \sum_{m=0}^{\lfloor \f{n}2 \rfloor} \sum_{k + j = n-2m} b_{k,j,n}^{\l,\mu}
  (\cos \theta \cos \phi)^k  \\
   & \qquad \times  (\sin \theta \sin \phi)^j
  D_{n-k-j}^{(\lambda + j, \mu +k)}(\cos \theta)
   D_{n-k-j}^{(\lambda + j, \mu +k)}(\cos \phi) 
  C_k^{\mu- \f12}(t) C_j^{\lambda- \f12}(s), \notag
\end{align}
where 
$$
 b_{k,j,n}^{\l,\mu} = \frac{\Gamma(\mu - \frac12)\Gamma(\lambda - \frac12)}{
     \Gamma(\lambda + \mu)} \frac{\Gamma(\lambda + \mu + k + j +1)}{  
  (n + \lambda + \mu) \Gamma(k+ \mu -\frac12)\Gamma(j+ \lambda-\frac12)}.
$$
\end{prop}

The identity \eqref{eq:addition} first appeared in \cite[p. 242, (4.7)]{Koor}, proved using a group theoretic method, 
but the constants were not given explicitly there. An analytic proof with explicit constants was
given in \cite[Theorem 2.3]{X97a}.  Our original proof uses the integral of \eqref{eq:addition} with respect to 
$w_{\l,\mu}(x) dx$ to prove
\begin{align*}
 & \frac{1}{(1-2 r s + r^2)^{\l+1}(1-2 r t + r^2)^{\mu+1} } \\
    & \qquad\qquad =  c_{\l+\f12,\mu+\f12} \int_{0}^1  \frac{1}{(1-2 r (y t + (1-y) s) + r^2)^{\l+\mu+2}}
  y^{\l}(1-y)^{\mu} dy,    
\end{align*}
which is equivalent to the case $d=2$ of \eqref{eq:main} and the case $d > 2$ follows from induction.

In the case of $\l_i =1$ for all $1 \le i \le d$, another elementary proof of \eqref{eq:main} can be deduced from the 
Hermite-Genocchi formula 
\begin{equation} \label{eq:HG}
  [x_1,\ldots, x_d] f =  \int_{\CT^d} f^{(d-1)} (x_1t_1+\ldots + x_d t_d) dt,
\end{equation}
where $[x_1,\ldots,x_d]f$ denotes the divided difference of $f$.  

\medskip\noindent
{\it Proof of Theorem \ref{thm:Gegen}}.
From \eqref{eq:main} with $d =2$, it follows that
$$
 \frac{1}{(1-2r x + r^2)^\l (1-2r y + r^2)^\mu} =  \s_{\l,\mu}
    \int_0^1 \frac{ s^{\l-1} (1-s)^{\mu-1} ds}{(1- 2 r (s x + (1-s) y) + r^2)^{\l+\mu}}
$$
for $0 \le r < 1$. Integrating with respect to $(1-y^2)^{\mu-1/2}$, we obtain, by the first identity of 
\eqref{eq:generatingC} that 
$$
  \frac{1}{(1-2r x + r^2)^\l}  = c_\mu \s_{\l,\mu} \int_{-1}^1
    \int_0^1 \frac{ s^{\l-1} (1-s)^{\mu-1} (1-y^2)^{\mu-\f12}ds dy}{(1- 2 r (s x + (1-s) y) + r^2)^{\l+\mu}}.
$$
Expanding both sides as power series of $r$, by the first identity of \eqref{eq:generatingC}, the 
identity \eqref{eq:Gegen-1} follows from comparing the coefficients of $r^n$. 

Now, replacing $\l$ by $\l+1$ in the last identity and multiplying by $1-r^2$, we obtain 
$$
  \frac{1-r^2}{(1-2r x + r^2)^{\l+1}}  = c_\mu \s_{\l+1,\mu} \int_{-1}^1
    \int_0^1 \frac{ (1-r^2) s^{\l} (1-s)^{\mu-1} (1-y^2)^{\mu-\f12}ds dy}{(1- 2 r (s x + (1-s) y) + r^2)^{\l+\mu+1}}.
$$
Expanding both sides as power series of $r$, by the second identity of \eqref{eq:generatingC}, the 
identity \eqref{eq:Gegen-2} follows from comparing the coefficients of $r^n$.
\qed

\section{Application to orthogonal polynomials of several variables}
\setcounter{equation}{0}

In the first subsection we recall basics on orthogonal polynomials of several variables and 
Fourier expansions in terms of them (cf. \cite{DX}). Product Gegenbauer polynomials are discussed 
in the second subsection and orthogonal polynomials on the unit ball are discussed in the third subsection. 

\subsection{Orthogonal polynomials of several variables} 
Let $W$ be a nonnegative weight function on a domain $\Omega$ of $\RR^d$, normalized so that  $\int_\Omega W(x) dx =1$.
Let $\CV_n^d(W)$ be the space of orthogonal polynomials of degree $n$ with respect to the inner product
$$
   \la f, g \ra_W : = \int_\Omega f(x) g(x) W(x) dx. 
$$
It is known that $r_n^d: = \dim \CV_n^d = \binom{n+d-1}{n}$. Let $\{P_j^n: 1 \le j \le r_n^d\}$ be an orthonormal basis of 
$\CV_n^d(W)$; that is, 
$$
     \int_\Omega P_j^n(x) P_k^m(x) W(x) dx = \delta_{j,k}\delta_{n,m}. 
$$ 
The Fourier coefficient $\wh f_j^n$ of  $f\in L^2(W,\Omega)$ is defined by $\wh f_j^n:= \int_\Omega f(x) P_j^n(x) W(x) dx$
and the Fourier orthogonal expansion of $f \in L^2(W,\Omega)$ is defined by 
$$
  f = \sum_{n=0}^\infty \proj_n f \quad \hbox{with}\quad \proj_n f(x) := \sum_{j=0}^n \wh f_j^n P_j^n(x). 
$$
The projection operator $\proj_n: L^2(W,\Omega) \mapsto \CV_n^d(W)$ can be written as
$$
  \proj_n f(x) = \int_{\Omega} f(y) P_n(x,y) W(y) dy  \quad \hbox{with}\quad 
  P_n(x,y) := \sum_{j=1}^{r_n^d} P_j^n(x) P_j^n(y), 
$$
where $P_n(\cdot,\cdot)$ is the reproducing kernel of $\CV_n^d(W)$. For $\d > 0$,
the Ces\`aro $(C,\d)$ means $S_n^\d f$ of the Fourier orthogonal expansion is defined by 
$$
  S_n^\d f := \f{1}{\binom{n+\d}{d}} \sum_{k=0}^n \binom{n-k+\d}{n-k} \proj_n f,
$$
which can be written as an integral of $f$ against the Ces\`aro $(C,\d)$ kernel  
$$
  K_n^\d(x,y) :=  \f{1}{\binom{n+\d}{d}} \sum_{k=0}^n \binom{n-k+\d}{n-k} P_k(x,y). 
$$
To emphasis the dependence on $W$, we will use notations such as $S_n^\d(W; f)$ and
$K_n^\d(W;\cdot,\cdot)$ in the rest of this section.

\subsection{Product Gegenbauer polynomials on the cube}

For $\lb = (\l_1,\ldots, \l_d)$, $\l_i > -\f12$, we consider the product Gegenbauer weight function 
$$
   W_\lb(x) = W_{\l,d}(x): =  c_\lb \prod_{i=1}^d w_{\l_i}(x_i), \qquad  x\in [-1,1]^d,  
$$
where $c_\lb = \prod_{i=1}^d c_{\l_i}$. It is easy to see that the product Gegenbauer polynomials are 
orthogonal polynomials and the reproducing kernel $P(W_\lb; \cdot,\cdot)$ of $\CV_n^d(W_\lb)$ is given by 
$$
P(W_\lb; x, y) =  \sum_{|\a| =n} \frac{1}{H_n^\l} P_\a(x) P_\a(y) \quad\hbox{with} \quad
      P_\a(x) := \prod_{i=1}^d C_{\a_i}^{\l_i} (x_i), 
$$ 
where $H_n^\l = \prod_{i=1}^d h_{\a_i}^{\l_i}$. The product formula of the Gegenbauer polynomials states that 
$$
 \frac{C_n^\l(x) C_n^\l (y)}{C_n^\l(1)} = c_{\l - \f12} \int_{-1}^1 C_n^\l(x y +\sqrt{1-x^2} \sqrt{1-y^2} t)(1-t^2)^{\l-1} dt, 
$$
which implies immediately that 
$$
  K_n^\d(W_\lb;x,y) = c_{\lb - \f{\one}{2}} \int_{[-1,1]^d} K_n^\d (W_\lb; z(x, y,t), \one) \prod_{i=1}^d (1-t_i^2)^{\l_i-1} dt_i,
$$
where $z(x, y,t) := (x_1y_1 + \sqrt{1-x_1^2} \sqrt{1-y_1^2} t_1, \ldots, x_d y_d+ \sqrt{1-x_d^2} \sqrt{1-y_d^2} t_d)$.
Below we deduce a closed form formula for $P_n(W_\lb; x,\one)$ and the $(C,\d)$ kernel $K_n^\d (W_\lb; x, \one)$. 
 
\begin{thm} \label{prop:prodcutC}
Let $\one = (1,\ldots,1)$. Then, 
\begin{equation}\label{eq:Pn-Gegen}
  P_n(W_\lb; x, \one) =  \sum_{m=0}^{\min \{\lfloor \frac{n}2 \rfloor, d-1\}}  (-1)^m \binom{d-1}{m}
     \s_{\lb} \int_{\CT^d} Z_{n-m}^{|\lb| + d-1} (\la x, y \ra) \prod_{i=1}^d y_i^{\l_i} dy,
\end{equation}
where $\s_\lb = \Gamma(|\lb|+d) /\prod_{i=1}^d \Gamma(\l_i+1)$. Furthermore, 
\begin{equation*}
   K_n^{d-2} (W_\lb; x, \one) = \frac{1}{\binom{n+ d-2}{n}} \sum_{m=0}^{\min \{n,d-1\}} \binom{d-1}{m} 
      \s_{\lb} \int_{\CT^d} Z_{n-m}^{|\lb| + d-1} (\la x, y \ra) \prod_{i=1}^d y_i^{\l_i} dy.
\end{equation*} 
\end{thm}

\begin{proof}
The identity \eqref{eq:main} implies that 
\begin{equation} \label{eq:prodCgen}
  \prod_{i=1}^d \frac{1}{(1- 2 r x_i + r^2)^{\l_i +1}} = \s_{\lb} \int_{\CT^d} \frac{1}{(1-2 \la x, y\ra + r^2)^{|\lb| +d}} 
      \prod_{i=1}^d y_i^{\l_i} dy.
\end{equation}
Multiplying by $(1-r^2)^d$ and applying \eqref{eq:generatingC}, we see that the left hand side can be expanded as
$$
    \prod_{i=1}^d \sum_{n=0}^\infty Z_{\a_i}^{\l_i} (x_i) r^n  = \sum_{n = 0}^\infty P_n(W_\lb;x,\one) r^n, 
$$
while the right hand side can be expanded, again by \eqref{eq:generatingC}, as 
\begin{align*}
  (1-r^2)^{d-1} &\sum_{n=0}^\infty  \s_{\lb+\one} \int_{\CT^d} Z_n^{|\lb|+d-1} (\la x,y \ra) \prod_{i=1}^d y_i^{\l_i} dy\, r^n \\
   & = \sum_{i=1}^{d-1}(-1)^i \binom{d-1}{i} \sum_{n=0}^\infty  \s_{\lb+\one} 
      \int_{\CT^d} Z_n^{|\lb|+d-1} (\la x,y \ra) \prod_{i=1}^d y_i^{\l_i} dy \, r^{n+2i}.
\end{align*}
The identity \eqref{eq:Pn-Gegen} follows from comparing the coefficient of $r^n$.

The second identity follows similarly. We multiply \eqref{eq:prodCgen} by $(1+r)^{d-1} (1-r^2)$ so that 
the left hand side can be expanded as
$$
   \frac1{(1-r)^{d-1}} \prod_{i=1}^d \sum_{n=0}^\infty Z_{\a_i}^{\l_i} (x_i) r^n  =
    \frac1{(1-r)^{d-1}} \sum_{n = 0}^\infty P_n(W_\lb;x,\one) r^n = \sum_{n=0}^\infty K_n^{d-2}(W_\lb;x,\one) r^n, 
$$
while the right hand side can be expanded, again by \eqref{eq:generatingC}, as 
\begin{align*}
  (1+r)^{d-1} &\sum_{n=0}^\infty  \s_{\lb+\one} \int_{\CT^d} Z_n^{|\lb|+d-1} (\la x,y \ra) \prod_{i=1}^d y_i^{\l_i} dy\, r^n. 
\end{align*}
Expanding $(1+r)^{d-1}$ in power of $r$, we can again compare the coefficient of $r^n$.
\end{proof}

Observe that the right hand side of \eqref{eq:Pn-Gegen} is a sum of at most $d$ terms. In the case of $\lb =0$, a closed
formula of $P_n(W_0; x, \one)$ was derived in \cite{X95} as a divided difference, which can be derived from 
\eqref{eq:Pn-Gegen} by applying the Hermite-Genocchi formula \eqref{eq:HG}. 

Let $k_n^\d(w_\l)$ denote 
the kernel of the $(C,\d)$ means of the Gegenbauer polynomials,
$$
  k_n^\d(w_\l;x,y) =  \f{1}{\binom{n+\d}{n}} \sum_{k=0}^n \binom{n-k+\d}{n-k} 
    \frac{1}{h_n^\l} C_k^\l(x) C_k^\l(y). 
$$
The proof of Proposition \ref{prop:prodcutC} also leads to the following proposition. 

\begin{prop} \label{prop:productCdelta}
For $\d > -1$, 
\begin{align*}
  K_n^{\d+d-1} (W_\lb; x, \one) = & \frac{1}{\binom{n+ \d + d-1}{n}} \sum_{m=0}^{\min \{n,d-1\}} \binom{d-1}{m} 
      \binom{n-m+\d}{n-m} \\
     & \times \s_{\lb+\one} \int_{\CT^d} k_{n-m}^\d(w_{|\lb| + d-1}; \la x, y \ra, 1) \prod_{i=1}^d y_i^{\l_i} dy.
\end{align*}
\end{prop}

One immediate consequence of the above relation shows that $K_n^\d(W_\k; x,y)$ is nonnegative if 
$\d \ge 2 (|\lb|+d) -1$, which was proved in \cite{LiXu} by a different method.

\subsection{Orthogonal  polynomials on the unit ball}
On the ball $\BB^d = \{x: \|x\|\le 1\}$ of $\RR^d$, we consider the weight function
$$
   W_{\l,\mu} (x) := b_{\l,\mu} \|x\|^{2\l} (1-\|x\|^2)^{\mu-\f12}, \quad \l \ge 0, \quad \mu > 0, \quad x \in \BB^d, 
$$
where $b_{\l,\mu}$ is a constant so that $b_{\l,\mu} \int_{\BB^d} W_{\l,\mu}(x) dx =1$. Let $\CH_m^d$ be
the space of spherical harmonics of degree $m$ in $d$ variables. Let $\s_m^d : = \dim \CH_m^d$ and 
let $\{Y_\nu^m: 1 \le \nu \le \s_m^d\}$ be an orthonormal basis of $\CH_m^d$. Define 
$$
   P_{j,\nu}^n(x) : = P_n^{(\mu-\f12, n-2j+\l+\f{d-2}{2})} (2\|x\|^2-1) Y_\nu^{n-2j}(x), 
$$
where $P_n^{(\a,\b)}(t)$ denotes the usual Jacobi polynomial of degree $n$. 
\begin{prop}
The set $\{P_{j,\nu}^n: 1 \le \nu \le \s_{n-2j}^d, 0 \le j \le n/2\}$ is a mutually orthogonal basis of $\CV_n^d(W_{\l,\mu})$
and the norm of $P_{j,\nu}^n$ in $L^2(W_{\l,\mu}, \BB^d)$ is given by 
$$
  H_j^n := \frac{ (\l+\f{d}{2})_{n-j} (\mu+\f12)_j (n-j+\l+\mu+\f{d-1}{2})}
      { j! (\l+\mu+\f{d+1}{2})_{n-j}  (n+\l+\mu+\f{d-1}{2})}, 
$$
where $(a)_n$ denotes the Pochhammer symbol, $(a)_n := a(a+1)\cdots (a+n-1)$.
\end{prop}

Using the spherical polar coordinates $x = r x'$, $0 \le r  \le 1$ and $x' \in \SS^{d-1}$, the case $\l =0$ 
was worked out explicitly in \cite{DX} and the general case follows similarly. 

In terms of this basis, the reproducing kernel $P_n(W_{\l,\mu};\cdot,\cdot)$ can be written as
\begin{align} \label{eq:Pn-ball}
  P_n(W_{\l,\mu}; x,y) = \sum_{0 \le j \le n/2} \sum_{\nu =1}^{\s_{n-2j}^d} \frac{1}{H_{j,n}}P_{j,\nu}^n(x)P_{j,\nu}^n(y). 
\end{align}
Our main result in this section is the following closed form of this kernel.

\begin{thm} \label{thm:repodBall}
For $\l > 0$ and $\mu > 0$, 
\begin{align}\label{eq:reprodBall}
P_n(W_{\l,\mu}; x,y)   = a_{\l,\mu} \int_{-1}^1 \int_0^1 \int_{-1}^1 
  & Z_n^{\l+\mu+\frac{d-1}{2}}  (\zeta(x,y, u,v,t)) (1-t^2)^{\mu-1} dt \\
  &   \times  u^{\l-1} (1-u)^{\f{d-2}{2}} du (1-v^2)^{\l-\f12} dv,\notag
\end{align}
where $a_{\l,\mu}$ is a constant such that the integral is 1 if $n =0$ and 
$$
 \zeta(x,y, u,v,t): = \|x\| \, \|y\| u v +  \la x, y \ra (1-u) + \sqrt{1-\|x\|^2} \sqrt{1-\|y\|^2}\, t ;
$$
furthermore, if $\l > 0$ and $\mu =0$, then 
\begin{align} \label{eq:repordBall2}
P_n(W_{\l,0}; x,y)   = a_{\l,0} \int_{-1}^1 \int_0^1
  &  \f12 \left[ Z_n^{\l+\frac{d-1}{2}}  (z(x,y, u,v, 1)) + Z_n^{\l+\frac{d-1}{2}}  (z(x,y, u,v, -1))\right] \notag \\
  &   \times  u^{\l-1} (1-u)^{\f{d-2}{2}} du (1-v^2)^{\l-\f12} dv. 
\end{align}
\end{thm}

\begin{proof}
Integrating \eqref{eq:addition} with respect to $w_{\mu -1}(t) dt$, then setting 
$\cos\t = \sqrt{1-\|x\|^2}$ and $\cos \phi = \sqrt{1-\|y\|^2}$, and replacing $\l$ by $\l + \f{d-1}{2}$, we deduce that 
\begin{align} \label{eq:addition-2}
 c_{\mu-\f12} \int_{-1}^1 & C_n^{\lambda+\mu+\f{d-1}{2}} ( \|x\| \|y\|s + \sqrt{1-\|x\|^2} \sqrt{1-\|y\|^2}  \,t) (1-t^2)^{\mu-1} dt \\
  = &  \sum_{j=0}^{\lfloor \f{n}2 \rfloor}  b_{0,n-2j,n}^{\l+\f{d-1}{2},\mu} \|x\|^{n-2j} \|y\|^{n-2j}
  D_{2j}^{(n-2j+\lambda +\f{d-1}{2}, \mu)}(\sqrt{1-\|x\|^2}) \notag \\
& \times   D_{2j}^{(n-2j+\lambda +\f{d-1}{2}, \mu)}(\sqrt{1-\|y\|^2})
  C_{n-2j}^{\l+ \f{d-2}{2}}(s). \notag
\end{align}
Let $x' = x/ \|x\|$ and $y' = y/ \|y\|$.  By the addition formula of spherical harmonics, 
$$
   \sum_{\nu =1}^{\s_m^d} Y_\nu(x) Y_\nu(y) = \|x\|^m \|y\|^m Z_m^{\f{d-1}{2}} (\la x',y' \ra). 
$$
Furthermore, by \cite[(2.1a)]{X97a}, we can deduce that 
\begin{align*}
 &  P_j^{(\mu-\f12, n-2j+\l+\f{d-2}{2})}(2\|x\|^2-1) =(-1)^j \sqrt{B_{j,n}} D_{2j}^{(n-2j+\l+\f{d-1}{2},\mu)}(\sqrt{1-\|x\|^2}),
\end{align*}
where 
$$
  B_{j,n} = \frac{\Gamma(n-2j+\l+\mu+\f{d+1}2) \Gamma(j+\mu+\f12)\Gamma(n-j+\l+\f{d}{2})}
  {\Gamma(\mu+\f12) \Gamma(n-2j+\l+\f{d}2)\Gamma(n-j+\l+\mu+\f{d-1}{2}) j! (n+\l+\mu+\f{d-1}{2})}.
$$
Substituting these two identities into the expression \eqref{eq:Pn-ball}, we obtain
\begin{align*}
 P_n(W_{\l,\mu}; x,y) = & \sum_{0 \le j \le n/2} \frac{B_{j,n}}{H_j^n} 
    D_{2j}^{(n-2j+\l+\f{d-1}{2},\mu)}(\sqrt{1-\|x\|^2}) \\
      & \times   D_{2j}^{(n-2j+\l+\f{d-1}{2},\mu)}(\sqrt{1-\|y\|^2}) 
    \|x\|^{n-2j} \|y\|^{n-2j} Z_{n-2j}^{\f{d-2}{2}}(\la x',y'\ra), 
\end{align*}
in which the last term can be replaced, according to \eqref{eq:Gegen-2}, by 
$$
   Z_{n-2j}^{\f{d-2}{2}}(\la x',y'\ra) = c \int_{-1}^1 \int_0^1 Z_{n-2j}^{\l+\f{d-2}{2}}(u v + (1-u) \la x',y'\ra)
     u^{\l-1}(1-u)^{\f{d-2}{2}} du\, w_\l(v) dv. 
$$
Now, a tedious verification shows that the constant 
$$
 \frac{B_{j,n}}{H_j^n} \frac{n-2j + \l+\f{d-2}{2}}{\l+\f{d-2}{2}} = \frac{n+\l+\mu+\f{d-1}{2}} {\l+\mu+\f{d-1}{2}}
 b_{0,n-2j,n}^{\l+\f{d-1}{2},\mu}.
$$
Consequently, \eqref{eq:reprodBall} follows from \eqref{eq:addition-2}. Finally, \eqref{eq:repordBall2} follows
from \eqref{eq:reprodBall} by taking the limit $\mu \to 0$. 
\end{proof}
 
\begin{cor}
For $\l > 0$ and $\mu > 0$, 
\begin{align}\label{eq:reprodBallat0}
P_n(W_{\l,\mu}; x, 0)   =  D_n^{(\l+\frac{d-1}{2},\mu)} (1)D_n^{(\l+\frac{d-1}{2},\mu)} (\sqrt{1-\|x\|^2}),
\end{align}
where $D_n^{(\l,\mu)}$ is the generalized Gegenbauer polynomial defined in \eqref{eq:genGegen}.
\end{cor}

\begin{proof}
Setting $y=0$ in \eqref{eq:reprodBall} shows that 
$$
P_n(W_{\l,\mu}; x, 0) = c_{\mu-\f12} \int_{-1}^1 Z_n^{\l+\mu+\frac{d-1}{2}}  (\sqrt{1-\|x\|^2} t) (1-t^2)^{\mu-1} dt, 
$$
from which the stated result follows from \cite[(2.11)]{X97a}.
\end{proof}

Taking the limit $\l =0$, the triple integrals of  \eqref{eq:reprodBall} becomes one layer, the resulted identity was
first proved in \cite{X99}, which plays an essential role in the study of Fourier orthogonal expansions with respect 
to the classical weight function $W_{0,\mu}$ on $\BB^d$. The closed form formula
in Theorem \ref{thm:repodBall} should play a similar role. We give one application.

For $f \in L^1(w_{\l+\mu+\f{d-1}{2}}; [-1,1])$ and $x, y \in \BB^d$,  define 
$$
  G_x f(y) :=  a_{\l,\mu} \int_{-1}^1 \int_0^1 \int_{-1}^1 f (\zeta(x,y, u,v,t)) (1-t^2)^{\mu-1} dt \\
   u^{\l-1} (1-u)^{\f{d-2}{2}} du w_\l(v) dv.
$$
As a consequence of Theorem \ref{thm:repodBall}, we can write
\begin{equation} \label{eq:cesaro}
   K_n^\d (W_{\l,\mu}; x,y) = G_x \left[k_n^\d(w_{\l+\mu+\f{d-1}{2}}; \cdot, 1)\right ](y), 
\end{equation}
where $k_n^\d(w_\l; s,t)$ denotes the Ces\`aro $(C,\delta)$ means of the Gegenbauer series. 

\begin{thm}
For $\l \ge 0$ and $\mu \ge 0$, the Ces\`aro $(C,\delta)$ means for $W_{\l,\mu}$ satisfy 
\begin{enumerate} [\quad 1.]
\item if $\d \ge 2 \l + 2 \mu + d$, then $S_n^\d(W_{\l,\mu}; f) \ge 0$ if $f(x) \ge 0$;
\item $S_n^\d(W_{\l,\mu}; f)$ converge to $f$ in $L^1(W_{\l,\mu}; \BB^d)$ norm or $C(\BB^d)$ norm 
if and only if $\d > \l+\mu+ \f{d-1}{2}$. 
\end{enumerate}
\end{thm}

\begin{proof}
The first assertion follows immediately from the non-negativity of the Gegenbauer series \cite{Gas}. 
To prove the second assertion, we first show that 
\begin{equation} \label{eq:intGx}
 b_{\l,\mu} \int_{\BB^d}  G_x g(y) W_{\l, \mu}(y) dy = c_{\l+\mu+\f{d-1}{2}} \int_{-1}^1 g(t) w_{\l+\mu+\f{d-1}{2}}(t) dt
\end{equation} 
for $g \in L^1(w_{\l+\mu+\f{d-1}{2}}; [-1,1])$. It suffices to prove it for $g$ being a polynomial, which
we can write as $g(t) = \sum_{k=0}^N \wh g_k Z_k^{\l+\mu+\f{d-1}{2}}(t)$, where $\wh g_0$ is exactly the
right hand side of \eqref{eq:intGx}. By Theorem \ref{thm:repodBall}, $G_x g(y) = 
 \sum_{k=0}^N \wh g_k P_k(W_{\l,\mu}; x,y)$, so that \eqref{eq:intGx} follows from the orthogonality 
 of $P_k(W_{\l,\mu}; x,\cdot)$. 
 A standard argument shows that $S_n^\d(W_{\l,\mu}; f)$ converges to $f$ in either $L^1(W_{\l,\mu}; \BB^d)$
norm or $C(\BB^d)$ norm if, and only if, 
$$
  \Lambda_n(x):= b_{\l,\mu} \int_{\BB^d} \left| K_n^\d(W_{\l,\mu};x,y)\right| W_{\l, \mu}(y) dy 
$$
is bounded, independent of $n$, for all $x \in \BB^d$. 
 From \eqref{eq:cesaro} and \eqref{eq:intGx}, we conclude that 
$$
\Lambda_n(x) \le c_{\l+\mu+\f{d-1}{2}} \int_{-1}^1\left|k_n^\d(w_{\l+\mu+\f{d-1}{2}}; t, 1)\right| w_{\l+\mu+\f{d-1}{2}}(t) dt,
$$
which is finite if $\d > \l+\mu+\f{d-1}{2}$ (Theorem 9.1.32 \cite[p. 246]{Szego}). Furthermore, by 
\eqref{eq:reprodBallat0}, we obtain 
$$
   K_n^\d(W_{\l,\mu}; 0, y) = k_n^\d\left(w_{\l+\f{d-1}{2},\mu}; \sqrt{1-\|x\|^2}, 1\right),
$$
where $w_{\l+\f{d-1}{2}, \mu}$ is the generalized Gegenbauer weight. Using spherical-polar coordinate and 
making a change of variable, it is easy to see that 
$$
 b_{\l,\mu} \int_{\BB^d} \left | K_n^\d(W_{\l,\mu}; 0, y) \right |W_{\l,\mu}(y) dy =
    c_{\l+\f{d-1}{2},\mu} \int_{-1}^1 \left| k_n^\d(w_{\l+\f{d-1}{2},\mu}; 1, t) \right | w_{\l+\f{d-1}{2},\mu}(t) dt,
$$ 
which is bounded if, and only if, $\d > \l + \mu + \f{d-1}{2}$ by \cite[Theorem 2.4]{DaiX}. 
\end{proof}

\medskip\noindent
{\bf Acknowledgement.} The author thanks Walter Van Assche for the elementary proof of 
\eqref{eq:main} and for giving him the permission to use it in place of the original proof.

\section*{Appendix. Alternative proof of Theorem \ref{thm:main}}

This appendix contains the original proof of Theorem \ref{thm:main}, which shows how the identity was discovered.
The proof is added here in respond to a request by a reader and it will appear only in the ArXiv version of this paper. 

\medskip\noindent
{\it Proof of Theorem \ref{thm:main}}.
Setting $x = \cos \t = \cos \phi$ in \eqref{eq:addition} and integrating with respect to $w_{\l,\mu}$, we obtain
by the orthonormality of $D_n^{(\l+j,\mu+k)}$ that 
\begin{align*}
c_{\l,\mu} \int_{-1}^1 C_n^{\l+\mu} & (x^2 t + (1-x^2) s) w_{\l,\mu}(x) dx \\
 & =   \sum_{m=0}^{\lfloor \f{n}2 \rfloor} \sum_{k + j = n-2m}b_{k,j}^n \frac{c_{\l,\mu}}{c_{\l+j,\mu+k}} 
     C_k^{\mu- \f12}(t) C_j^{\lambda- \f12}(s).
\end{align*}
Changing variable $y = x^2$ in the integral and simplifying the constants in the right hand side, the above 
identity can be written as 
$$
 c_{\l,\mu} \int_{0}^1 Z_n^{\l+\mu} (y t + (1-y) s) y^{\l-\f12}(1-y)^{\mu -\f12} dy 
    =   \sum_{m=0}^{\lfloor \f{n}2 \rfloor} \sum_{k + j = n-2m} Z_k^{\mu-\f12}(t)Z_j^{\l-\f12}(s).
$$ 
In particular, using the identity 
$$
    \sum_{m=0}^{\lfloor \f{n}2 \rfloor} \sum_{k + j = n-2m} a_{k,j} + 
     \sum_{m=0}^{\lfloor \f{n-1}2 \rfloor} \sum_{k + j = n-1-2m}a_{k,j}
 =  \sum_{m=0}^n \sum_{k + j = m} a_{k,j},
$$
and replacing $\l$ by $\l+1/2$ and $\mu$ by $\mu +1/2$, we deduce that 
\begin{align*}
   \sum_{m=0}^{n} & \sum_{k + j = m} Z_k^{\mu}(t)Z_j^{\l}(s) \\
      & =  c_{\l+\f12,\mu+\f12} \int_{0}^1 \left[ Z_n^{\l+\mu+1} (y t + (1-y) s) +
  Z_{n-1}^{\l+\mu+1} (y t + (1-y) s)\right]  y^{\l}(1-y)^{\mu} dy. 
\end{align*}
Next we multiply the above identity by $r^n$, $0 \le r < 1$, and summing up over $n$. In 
the left hand side, we obtain
\begin{align*}
   \sum_{n=0}^\infty & \sum_{m=0}^{n}   \sum_{k + j = m} Z_k^{\mu}(t)Z_j^{\l}(s)  r^n = \frac{1}{1-r} 
      \sum_{n=0}^\infty  \sum_{k + j = n} Z_k^{\mu}(t)Z_j^{\l}(s)  r^n  \\
       & =  \frac{1}{1-r}  \sum_{k=0}^\infty Z_k^{\mu}(t) r^k \sum_{j=0}^\infty  Z_j^{\l}(s)  r^j
        = \frac{(1+r) (1-r^2)}{(1-2 r s + r^2)^{\l+1}(1-2 r t + r^2)^{\mu+1} }
\end{align*} 
by \eqref{eq:generatingC}. The right hand side can be summed up by using 
$$
   \sum_{n=0}^\infty \left[ Z_n^{\l+\mu+1}(u) + Z_{n-1}^{\l+\mu+1}(u) \right] r^n 
      = (1+r)   \sum_{n=0}^\infty  Z_n^{\l+\mu+1}(u) r^n 
    =  \frac{(1+r) (1-r^2)}{(1-2 r u + r^2)^{\l+\mu+2}}.
$$
Putting these together, we have proved that 
\begin{align*}
 & \frac{1}{(1-2 r s + r^2)^{\l+1}(1-2 r t + r^2)^{\mu+1} } \\
    & \qquad\qquad =  c_{\l+\f12,\mu+\f12} \int_{0}^1  \frac{1}{(1-2 r (y t + (1-y) s) + r^2)^{\l+\mu+2}}
  y^{\l}(1-y)^{\mu} dy.    
\end{align*}
Rescaling and replacing $2 r /(1+r^2)$ by $r$, it follows that 
$$
   \frac{1}{(1- r s)^{\l+1}(1- r t)^{\mu+1} } =  c_{\l+\f12,\mu+\f12} \int_{0}^1  \frac{1}{(1- r (y t + (1-y)s ))^{\l+\mu+2}}
  y^{\l}(1-y)^{\mu} dy,
$$ 
which proves \eqref{eq:main} for $d=2$ and $\l,\mu > 1/2$ when we replace $\l$ by $\l -1$ and $\mu$ by $\mu -1$. 
Analytic continuation shows that the identity holds for $\l, \mu > 0$. 

The general case of \eqref{eq:main} follows from induction. Assume that \eqref{eq:main} has been established 
for $d$ variables. Let the constant in front of the integral in \eqref{eq:main} be denoted by $\s_\lb$, its value can be 
determined by setting $r=0$ and is not important for the induction. For $x \in \RR^{d+1}$, write $x = (x_1,x')$ and
$\lb = (\l_1,\lb')$. Then, using \eqref{eq:main} for $d=2$,
\begin{align*}
 \prod_{i=1}^{d+1} \f 1 {(1 -  r x_i)^{\l_i}} & = \s_{\lb'}  \frac{1}{(1-r x_1)^{\l_1}} \int_{\CT^d} \frac{ \prod_{i=2}^{d+1} u_i^{\l_i-1}}
   {(1- r \la x',u\ra)^{|\lb '| }} du \\
    & = c \int_{\CT^d} \int_0^1 \frac{ \prod_{i=2}^{d+1} u_i^{\l_i-1} (1-t_1)^{|\lb'|} t_1^{\l_1}}
      {(1- ( (1-t_1)\la x',u\ra + t_1 u_1 ) r )^{|\lb | }} dt_1 du,
\end{align*}
where we have written $u =(u_2,\ldots, u_{d+1}) \in \CT^d$. Since $|u| = 1$, we write $u_{d+1} = 1-u_2-\cdots -u_d$ and 
make a change of variables $t_i = (1-t_1) u_i$ for $i =2,\ldots, d$, it follows that 
$$
  (1- t_1)^{|\lb'|} t_1^{\l_1} \prod_{i=2}^{d+1} u_i^{\l_i} =
      \prod_{i=1}^d t_i^{\l_i} (1-t_1-\cdots - t_{d+1})^{\l_{d+1}} = \prod_{i=2}^{d+1} t_i^{\l_i},   
$$
where $t_{d+1} = 1-t_1-\cdots - t_d$ and, moreover,  $(1- t_1) \la x',u\ra + t_1 u_1 =  \la x, t \ra$ with $x = (x_1,x')$,
$t = (t_1,\ldots, t_{d+1})$ and $t_{d+1} = 1-t_1-\cdots - t_d$. Consequently, since $(1-t_1)^{d-1} d u = d t_2\cdots d t_d$, 
we conclude that 
\begin{align*}
 \prod_{i=1}^{d+1} \f 1{(1 -  r x_i)^{\l_i}} = c \int_0^1   \int_{\CT_{1-t_1}^d} 
     \frac{ \prod_{i=1}^{d+1} t_i^{\l_i-1}} 
      {(1- \la x, t \ra r)^{|\l | }} dt_1\cdots d t_{d+1} =  c \int_{\CT^{d+1}}  \frac{ \prod_{i=1}^{d+1} t_i^{\l_i-1}} 
      {(1- \la x, t \ra r)^{|\l | }} dt,
\end{align*}
where $c$ is a constant and its value can be determined by setting $r = 0$. This completes the proof of 
\eqref{eq:main}. \qed
\medskip

\end{document}